\documentclass[12pt, a4paper]{article}

\usepackage{amsmath,amssymb,amsthm}
\usepackage{tikz}
\usepackage{float}
\usepackage[utf8]{inputenc}
\usepackage[english]{babel}
\usepackage[T1]{fontenc}
\usepackage{lmodern}
\usepackage[totalwidth=160mm, totalheight=247mm]{geometry}

\usepackage[auth-sc, affil-it]{authblk}

\renewcommand{\leq}{\leqslant}
\renewcommand{\geq}{\geqslant}

\newcommand{\bbZ}{\mathbb{Z}}

\newcommand{\bbR}{\mathbb{R}}
\newcommand{\bbS}{\mathbb{S}}
\newcommand{\bbC}{\mathbb{C}}
\newcommand{\bbD}{\mathbb{D}}
\DeclareMathOperator{\diam}{diam}
\DeclareMathOperator{\piz}{\pi_0}

\newcommand{\Mc}{\mathcal{M}}

\theoremstyle{plain}
\newtheorem{theo}{Theorem}[section]
\newtheorem{prop}[theo]{Proposition}
\newtheorem{lemm}[theo]{Lemma}
\newtheorem{coro}[theo]{Corollary}
\theoremstyle{definition}
\newtheorem{defi}[theo]{Definition}
\newtheorem{exam}[theo]{Example}
\newtheorem{rema}[theo]{Remark}

\newtheorem{ques}{Question}
\newtheorem*{clai-nn}{Claim}
\newcommand{\myvcenter}[1]{\ensuremath{\vcenter{\hbox{#1}}}}

\title{\textbf{A numerical scale for non-locally connected planar continua}}

\author[1,2]{Timo Jolivet\thanks{The first author was supported by Agence Nationale de la Recherche
through the project FAN -\emph{Fractals and Numeration} - ANR-12-IS01-0002.}}
\author[3,4]{Beno\^it Loridant\thanks{The second author was supported by the Austrian Science Fund through the FWF Project 22 855 - \emph{Topology of fractal tiles} - and the project FAN -\emph{Fractals and Numeration} - FWF I1136.}}
\author[5]{Jun Luo\thanks{The third author was supported by the Chinese National Natural Science Fundation Project 10971233.}}

\affil[1]{
    Universit\'e Paris Diderot,
    LIAFA Case 7014,
    75205 Paris Cedex 13, France
}
\affil[2]{
    Department of Mathematics,
    University of Turku 20014, Finland
}
\affil[3]{Corresponding author: loridant@dmg.tuwien.ac.at
}
\affil[4]{  Montanuniversit\"at Leoben,
    Franz Josefstrasse 18, Leoben 8700, Austria 
}
\affil[5]{
    School of Mathematics and Computational Science,
    Sun Yat-Sen University, Guangzhou 512075, China
}

\date{\today}

\begin{document}
\maketitle

\begin{abstract}
We introduce a numerical scale to quantify to which extent a planar continuum is not locally connected. For a locally connected continuum, the numerical scale is zero; for a continuum like the topologist's sine curve, the scale is one; for an indecomposable continuum, it is infinite. Among others, we shall pose a new problem that may be of some interest: can we estimate the scale from above for the Mandelbrot set $\Mc$ ?

\end{abstract}

\section{Introduction}
\let\thefootnote\relax\footnote{Part of our study was also supported by the FWF Project SFB F50 of the Austrian Science Fund.}
This paper is about planar continua, {\em i.e.}, compact connected sets in the complex plane $\bbC$ or the extended complex plane $\hat{\bbC}$. By the Hahn-Mazurkiewicz theorem, such a continuum is locally connected if and only if it is the image of $[0,1]$ under a continuous map $g:[0,1]\rightarrow\bbC$. In particular, a locally connected continuum is path-connected. 

In the present article, we are interested in measuring how far a planar continuum is from being locally connected. To this effect, we will introduce ``a numerical scale'' which ``quantifies'' the extent to which a planar continuum is not locally connected.

Our work is motivated by possible applications in complex dynamics. After Douady and Hubbard \cite{DH-82,DH-84,DH-85}, complex dynamics becomes ``a focus of interest'' in the late 1980's \cite[page (v)]{CG-93}. The study of iterated polynomials provides many  examples of planar continua. In the dynamical plane, Julia sets $J_c$ of hyperbolic or parabolic polynomials $f_c(z)=z^2+c$ are locally connected; however, $J_c$ is not locally connected if $f_c$ has an irrationally neutral fixed point that does not correspond to a Siegel disk (\cite[Theorem V.4.4]{CG-93}). In the parameter plane, local connectedness of the Mandelbrot set $\Mc$ remains unknown and has been one of the most central problems in complex dynamics.

More recently, Hubbard and Schleicher (\cite[Theorem 6.2]{HS-2012}) study the multicorns $\Mc_d^*$, {\em i.e.}, the set of parameters $c$ for which the Julia set of the anti-holomorphic polynomial $z\mapsto \overline{z}^n+c$ is connected. They show that the multicorns are not path-connected and, hence, not locally connected. The multicorns thus give new examples of planar continua that are not locally connected.

Therefore, a problem of interest will be to estimate the numerical scale we introduce here for typical planar continua, like the Mandelbrot set $\Mc$, multicorns or Julia sets of infinitely renormalizable polynomials $z\mapsto z^2+c$ which are not locally connected.

The paper is organized as follows. In Section~\ref{sec:mainresult}, we recall the notion of \emph{fibers}, our main tool, and define the numerical scale. In Section~\ref{sec:basicprop}, we give some basic properties of fibers. In Sections~\ref{sec:proofmain} and~\ref{sec:DetailedProof}, we discuss the relation between trivial fibers and local connectedness. Finally, we collect in Section~\ref{sec:questions} several questions that may merit some attention.\\

\textbf{Acknowledgments.} The authors are grateful to J\"org Thuswaldner for his suggestions that improved the readability of this paper.
\section{Notions, Main result and Examples}\label{sec:mainresult}

Given a continuum $K\subset\bbC$ with $\bbC\setminus K$ connected, let $\varphi$ be the Riemann mapping sending the exterior of $K$ to the exterior of the unit disk that is normalized so as to fix $\infty$ with positive real derivative. For any $\theta\in\bbS^1=\bbR/\bbZ$, $\displaystyle R_K(\theta)=\varphi^{-1}\left(\left\{re^{2\pi{\bf i}\theta}: r>1\right\}\right)$ is called the {\em external ray of $K$ at angle $\theta$}. An external ray $R_K(\theta)$ is said to {\em land} if its limit set $\overline{R_K(\theta)}\cap K$ is a single point.

In~\cite{Sch99-a} Schleicher fixes a countable set of angles $Q\subset\bbS^1$ such that all the external rays at angles in $Q$ land and defines a {\em separation line} (with respect to $Q$) as: (1) either the closure of the union of two external rays with angles in $Q$ which land at a common point on $K$, (2) or the closure of the union of two such rays which land at different points on $K$, together with a simple curve in the interior of $K$ connecting the two landing points. 
Then, two points $z,z'\in K$ are said to be {\em separated from each other} if there is a separation line $\gamma$ avoiding $z$ and $z'$ such that these two points are in different components of $\bbC\setminus\gamma$. For any point $z\in K$, the component of
$\left\{z'\in K:\ z' \ \text{can\ not\ be\ separated\ from}\ z\right\}$
containing $z$ is called the {\em fiber of $z$}, which is also a continuum whose complement is connected \cite[Lemma 2.4]{Sch99-a}. 

Among other fundamental properties of fibers, Schleicher further shows that if the fiber at $z\in K$ is trivial, {\em i.e.}, consists of a single point, then $K$ is locally connected at $z$~\cite[Proposition 2.9]{Sch99-a}. Moreover, under three additional assumptions on external rays $R_K(\theta)$ with $\theta\in Q$, he also shows that if $K$ is locally connected then all its fibers are trivial~\cite [Proposition 2.10]{Sch99-a}. 

The choice of $Q$ is important in Schleicher's definition of fibers, especially in the study on fibers of Julia sets \cite{Sch99-b} and the Mandelbrot set \cite{Sch04}. 

We will remove the dependence on the choice of $Q\subset\bbS^1$ and define fibers for planar continua whose complement has finitely many components. In this more general setting, we can establish without any further assumptions an equivalence between  ``trivial fibers'' and local connectedness of $K$. 

Following the basic philosophy of ``fibers of fibers'', we continue to define a new numerical scale which quantifies the ``non-local connectedness'' or  the ``deviation from local connectedness'' in a reasonable way.

Throughout this paper, $\piz(X)$ denotes the collection of components of a set $X\subset\bbC$.

\begin{defi}[Good cuts and fibers]
Let $X \subseteq \bbC$ be a continuum such that $\piz(\bbC \setminus X)$ is finite.
\begin{itemize}
\itemsep=0pt
\item     A \emph{good cut of $X$}  is a simple closed curve $\gamma : [0,1] \rightarrow \hat{\bbC}$ such that
    $\gamma \cap \partial X$ is a nonempty finite set
    and such that $\gamma \setminus X$ is nonempty. 

\item Two points $x,y \in X$ are \emph{separated} by a good cut $\gamma$ of $X$
    if they belong to distinct components of $\bbC \setminus \gamma$.

\item  The \emph{pseudo-fiber} $E_x$ of $x \in X$  is the set of all the points $y \in X $ such that $x$ and $y$ are not separated by any good cut of $X$.
    The \emph{fiber} $F_x$ is the component of $E_x$ that contains $x$.
    We say $E_x$ or $F_x$ is \emph{trivial} if it consists of $\{x\}$ only.
\end{itemize}
\end{defi}

We can now state the main result of this paper.
\begin{theo}\label{main}
Let $X \subseteq \bbC$ be a continuum such that $\pi_0(\bbC \setminus X)$ is finite. Then the following assertions are equivalent.
\begin{itemize}
\item[$(i)$] $X $ is locally connected.
\item[$(ii)$] For each $x \in X$, the pseudo-fiber $E_x$ is trivial.
\item[$(iii)$]  For each $x \in X$, the fiber $F_x$ is trivial.
\end{itemize}
\end{theo}

This leads us to the definition of a numerical scale for planar continua.
\begin{defi}[Higher order fibers]\label{def:order}
We define \emph{higher order fibers} of a continuum $X\subset\mathbb{C}$ with finitely many complementary components by induction.
\begin{itemize}
\item A \emph{fiber of order $0$} is $X$ itself.
\item A \emph{fiber of order $k\geq 1$} is a fiber of the subcontinuum $Y \subseteq X$,
where $Y$ is a fiber of order $k-1$.
\end{itemize}
\end{defi}
\begin{rema}
The consistency of this definition will be the purpose of Proposition~\ref{prop:fibers}: every fiber of a continuum $X\subset\mathbb{C}$ with finitely many complementary components is again a continuum with finitely many complementary components.
\end{rema}

\begin{defi}[Numerical scale]\label{def:numscale}
Let  $X \subseteq \bbC$ be a continuum such that $\pi_0(\bbC \setminus X)$ is finite.
We define $\ell(X)$ as the smallest integer $k$ such that for all $x\in X$, there exist
$$N_0=X\supset N_1\supset\cdots N_{p-1}\supset N_p=\{x\}
$$
for some $p\leq k+1$, where $N_i$ is a fiber of $N_{i-1}$.
If such an integer $k$ does not exist, we write $\ell(X) = \infty$.
\end{defi}

\begin{rema}
The fibers of $X$ are the fibers of order $1$
and $X$ is locally connected if and only if $\ell(X) = 0$ by Theorem~\ref{main}.
\end{rema}

\begin{rema}
Kiwi~\cite{Kiwi-04} considers Julia sets $J$ and defines a modified notion of fibers, based on the topology of $J$ and without mentioning its embedding in $\bbC$. Actually, given a monic polynomial $f$ and its Julia set $J(f)$, let $J_{fin}(f)$ denote the set formed by all periodic and preperiodic points in $J(f)$ which are not in the grand orbit of a Cremer point. If $C$ is a component of $J(f)$ and $Z$ is a finite subset of $J_{fin}(f)$, then $C\setminus Z$ has finitely many components, each of which is an open subset of $C\setminus Z$ \cite[Proposition 2.13]{Kiwi-04}. Given $z\in J(f)$, Kiwi defines the fiber of $z$, denoted as ${\rm Fiber}(z)$, to be the set of points $z'$ that lie in the same component of $J(f)\setminus Z$ for all finite subsets $Z$ of $J_{fin}(f)$ with $\{z,z'\}\cap Z=\emptyset$. We note that if $J(f)$ is connected then $ F_z\subset {\rm Fiber}(z)$, where $F_z$ is the fiber of $J(f)$ at $z$ defined in our paper. The converse containment is currently unclear and suggests an interesting topic that deserves a further study. For instance, we will wonder whether we can extend the notions of fiber and scale so that  (1) Kiwi's and Schleicher's definitions may be included as special cases of ours and (2) more general continua than planar ones can be discussed.
\end{rema}

\begin{exam}[Nontrivial fibers]
Let
\begin{align*}
X & = \{(t,0) : t \in [0,1]\} \ \cup \ \underbrace{\{(0,t) : t \in [0,1]\}}_{\displaystyle =:X_0} \\
  & \ \cup \ \bigcup_{n=0}^\infty \{(2^{-n},t) : t \in [0,1]\} \ \subseteq \ \bbR^2 \\
  & = \ \myvcenter{
\begin{tikzpicture}
\draw[thick] (0,0) -- (5,0);
\foreach \i in {0,...,8}
{
    \draw[thick] (5/2^\i,0) -- (5/2^\i,1);
}
\draw[thick] (0,0) -- (0,1);
\end{tikzpicture}
}\,.
\end{align*}
The pseudo-fiber of $x \in X$ is nontrivial if and only if $x \in X_0$,  because a good cut $\gamma$ can intersect the set $\partial X=X$ only finitely many times.  For each $x\in X_0$,  $E_x = X_0=F_x $. As a consequence, $\ell(X)=1$.
\end{exam}

\begin{exam}[$E_x \neq F_x$]
Let
$$S = \{(t-1, \pm t) : t \in [0,1]\}\;\textrm{ and }\; T = \{(1-t, \pm t) : t \in [0,1]\}.
$$
We define
\begin{align*}
X =& \underbrace{\{(s,t) \in \left [-1,1 \right ]^2 : |s| + |t| \leq 1 \}}_{\displaystyle =:L}
  \ \cup \ \{(t,0) : 1<|t|\leq 2\}\ \cup\\
  & \bigcup_{n=0}^\infty \left(\left\{x-\left(2^{-n},0\right): x\in S\right\}\cup \left\{x+\left(2^{-n},0\right): x\in T\right\}\right) \\
  & =
\myvcenter{
\begin{tikzpicture}[x=2cm,y=2cm]
\draw[thick] (-2,0) -- (-1,0);
\draw[thick] (1,0) -- (2,0);
\foreach \i in {0,...,8}
{
    \draw[thick] (1/2^\i,1) -- (1+1/2^\i,0);
    \draw[thick] (1/2^\i,-1) -- (1+1/2^\i,0);
    \draw[thick] (-1/2^\i,1) -- (-1-1/2^\i,0);
    \draw[thick] (-1/2^\i,-1) -- (-1-1/2^\i,0);
}
\fill[black] (-1,0) -- (0,-1) -- (1,0) -- (0,1);
\end{tikzpicture}
}\,.
\end{align*}
In the above picture, $L$ is the dark square, $S\cup T$ its boundary, and the infinite union is made of the whiskers.
Now, for $x \in X \setminus L$ there exists a simple closed curve $\gamma$ around $x$ such that $\gamma \cap \partial X$ is finite,
so $E_x$ and $F_x$ are trivial. We list the different pseudo-fibers and fibers in the following table.

$$\begin{array}{c|c|c}
x\in\cdots&E_x&F_x\\\hline
X\setminus L&\{x\}&\{x\}\\
L\setminus (S\cup T)&\{(0,1),(0,-1),x\}&\{x\}\\
S\setminus\{(0,1),(0,-1)\}&S&S\\
T\setminus\{(0,1),(0,-1)\}&T&T\\
\{(0,1),(0,-1)\}& L&L \\\hline
\end{array}$$
It follows that, in this example, we again have $\ell(X)=1$.
\end{exam}

\begin{rema}We see in this example a consequence of the condition $\gamma\setminus X \neq \varnothing$ in the definition of a good cut $\gamma$ of $X$. Indeed, if we allowed  $\gamma\subset X$, then the point $(0,1)$ could be separated by a good cut from any point $y$ in $L^o$. Thus we would obtain $E_{(0,1)}=F_{(0,1)}=S\cup T$. With our definition, we have in this example $F_{(0,1)}=L$.
\end{rema}

\begin{exam}\label{CantorEx}
We take two copies of the middle third Cantor set placed at levels $y=\pm1$ and consider $X$ as the union of segments through the point $O=(1/2,0)$ that connect the symmetric points of the Cantor sets about $O$. See the figure below. 
Then,
$F_{O}=X$, while for all $x\ne O$, $F_x$ is a segment containing $x$. Now, if $x\in X$, choosing $N_0=X$, $N_1$ a segment from $O$ to a point on the Cantor sets that contains $x$, and $N_2=\{x\}$ in the definition of the numerical scale, we see that we have $\ell(X)=1$.
\begin{center}
    \includegraphics[height=5cm]{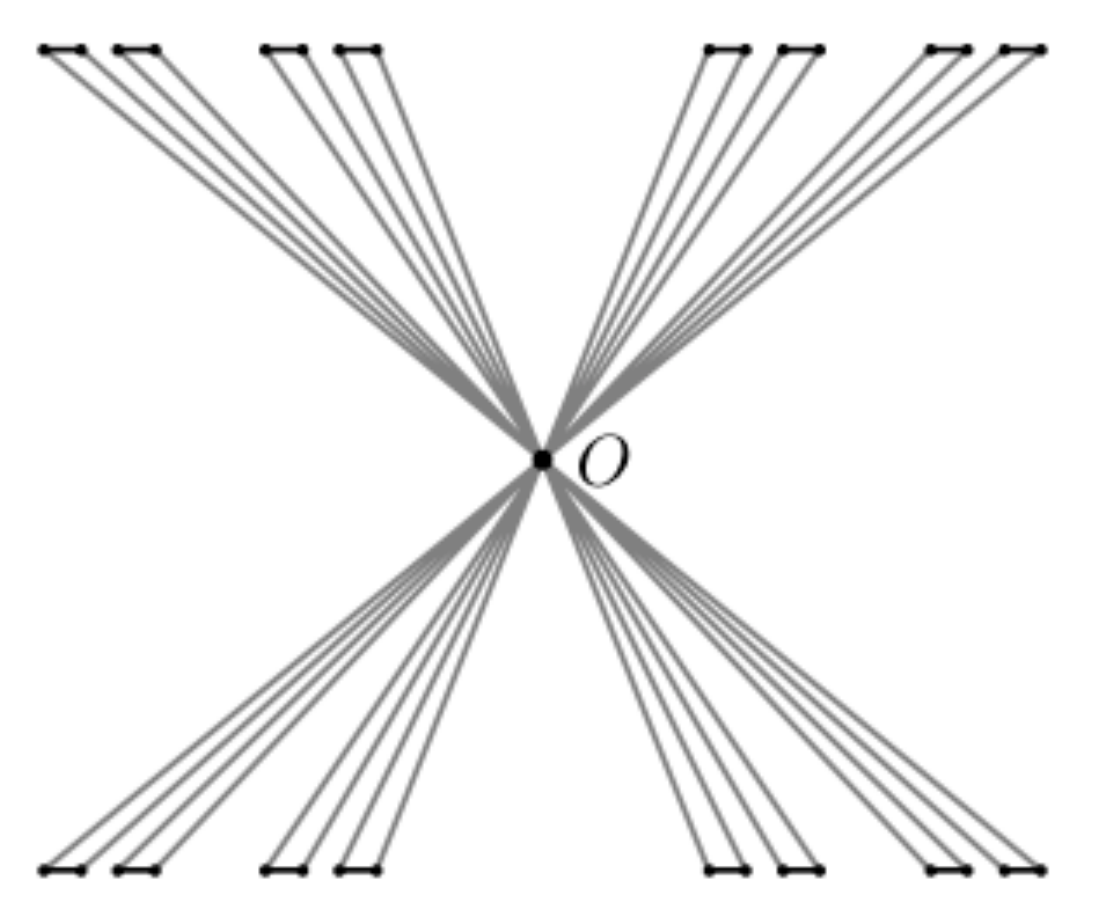}
\end{center}

\end{exam}

In the forthcoming examples, the boundary of the space $X$ has a remarkable property: it contains an indecomposable space. Thus, for $X$ as in Example~\ref{indec1}, we have $\ell(X)=\infty$. We recall the notion of indecomposable space.

\begin{defi}[Indecomposable space, composant] A non-degenerate topological space $X$ is \emph{indecomposable} if it is connected and whenever $X=A\cup B$ with $A,B$ connected, closed subsets of $X$, then $A=X$ or $B=X$ (see~\cite[Section 43, Chapter V]{Kur68}). Moreover, given a non-degenerate continuum $X\subset\bbC$ and a point $x\in X$, the \emph{composant} of $X$ containing $x$ is the union of all the proper sub-continua $N\subset X$ with $x\in N$.
\end{defi}
\begin{lemm} The intersection of an indecomposable space $X$ with a simple closed curve $\gamma$ is either empty or an uncountable set.
\end{lemm}
\begin{proof}
It is known that $X$ has uncountably many composants if it is indecomposable~\cite[Section 48, Chapter V, Theorem 7]{Kur68}. Now, suppose that $X\cap \gamma=\{x_1,\ldots,x_k\}$ is finite. For $i=1,\ldots,k$,  we call $C_i$ the composant of $X$ with $x_i\in C_i$. Then for all $i$, $C_i$ contains all the connected components $P$ of $X\setminus \gamma$ whose closure $\overline{P}$ contains $x_i$. Indeed, such a $\overline{P}$ is a proper subcontinuum of $X$ containing $x_i$. Now, it follows that $X=\cup_{i=1}^k C_i$, hence $X$ has at most finitely many composants, which is impossible since $X$ is indecomposable.
\end{proof}

\begin{exam}\label{indec1} Let $X\subset\bbC$ be a continuum and $M\subset \partial X$ an indecomposable continuum. If $J$ is a good cut of $X$ enclosing a point $x\in M$ then $M\cap J$ a finite set, and hence is empty. Thus $M$ is enclosed  by $J$ and $M\subset E_x$ for each $x\in M$; moreover, connectedness of $M$ indicates that $M\subset F_x$. This gives a typical example of non-trivial fiber. We even have $\ell(X)=\infty$.
\end{exam}

\textbf{Particular case (1) of Example~\ref{indec1}.} Let $X$ be the Brower-Janiszewski-Knaster continuum, also called  \emph{buckethandle}, depicted for example in~\cite[Section 43, Chapter V]{Kur68}.  Then $X^o=\varnothing$ and $X$ is neither path-connected nor locally connected. With our definition, we have $E_x=F_x=X$ for all $x\in X$ and $\ell(X)=\infty$.

\textbf{Particular case (2) of Example~\ref{indec1}.}  Let $X$ be a \emph{Wada lake} together with its boundary. Then $X^o\ne\varnothing$ is an open disk, and its boundary $\partial X$ is indecomposable. We also have here for all $x\in X$ that $F_x=X=E_x$ and $\ell(X)=\infty$.
\vspace{0.2cm}

\section{Basic properties of fibers}\label{sec:basicprop}

We collect here some basic properties concerning fibers. \\

The following theorems of Torhorst and Cartheodory~\cite{Car13,Car13b,Mil99} will be used frequently in our paper. 
\begin{theo}\label{TorTheo} \emph{(Theorem of Torhorst, see~\cite[Part B, Section VI, Torhorst Theorem and Lemma 2]{Why79})}. The boundary $\partial C$ of each component $C$ of the complement of a locally connected continuum $X$ is itself a locally connected continuum. Moreover, if $X$ has no cut point, then $\partial C$ is a simple closed curve.
\end{theo}
\begin{theo}\label{CaraTheo}\emph{(Theorem of Caratheodory, see~\cite[Theorem 9.8]{Pom75})}.
Let $\varphi:\Delta:=\{z\in\bbC:|z|>1\}\cup\{\infty\}\to \hat\bbC$ with $\varphi(\infty)=\infty$ be univalent, \emph{i.e.}, holomorphic and one-to-one. Set $G=\varphi(\Delta)$. Then the following assertions are equivalent.
\begin{itemize}
\item[(i)] $\varphi$ has a continuous extension to $\overline{\Delta}=\{z\in\mathbb{C}:|z|\geq 1\}\cup\{\infty\}$.
\item[(ii)] $\partial G$ is locally connected.
\item[(iii)] $\hat\bbC\setminus G$ is locally connected.
\end{itemize}
\end{theo}

In Example 2.3, a point $x\in X^o$ can be separated from  $y=(0,1)\in X$ by a simple closed curve $\gamma\subset X^o$, while $x$ and $y$ cannot be separated by a good cut of $X$.  This cannot happen if $X$ is locally connected, as we show in Proposition~\ref{interior-J}.
\begin{prop}\label{interior-J}
Let $X\subset\bbC$ be a locally connected continuum such that $\bbC\setminus X$ has finitely many components $U_1,\ldots, U_k$. Suppose that two points $x,y\in X$ are separated by a simple closed curve $\gamma\subset X^o$.  Then they are even separated by a good cut of $X$. 
\end{prop}
\begin{proof}
Denote by $W$ the component of $X^o$ containing $\gamma$ and assume that $x\in \textrm{Int}(\gamma)$ and $y\in \textrm{Ext}(\gamma)$. We may further assume that $\gamma=\{z: |z|=1\}$, by the theorem of Sch\"onflies (see~\cite{Thom92}).

Choose $k+1$ external rays $R_j=\{re^{2\pi{\bf i}\theta_j}: r\ge 1\}$ for $\theta_j=\frac{j}{k+1}$ with $1\le j\le k+1$. For $j=1,\cdots,k+1$, let us denote by $r_j>1$ the radius such that $z_j=r_je^{2\pi{\bf i}\theta_j}$ is the first point of $R_j$ that belongs to $\partial W$, meaning that the half-open arc $\alpha_j=\left\{re^{2\pi{\bf i}\theta_j}: 1\le r<r_j\right\}$ is entirely contained in $W$. Since $\partial W$ is contained in $\bigcup_{i=1}^k\partial U_i$ we know that two of the $k+1$ points $z_j$, say $z_{j_1}$ and $z_{j_2}$, belong to $\partial U_i$ for some fixed $i\in\{1,\ldots,k\}$.

By connectedness of $X$, we can see that $U_i$ is a simply connected domain; on the other hand, by the theorem of Torhorst (see Theorem~\ref{TorTheo}),  $\partial U_i$ is a locally connected continuum. Therefore, $z_{j_1},z_{j_2}\in\partial U_i$ can be connected by an open arc $\alpha\subset U_i$ (see Theorem~\ref{CaraTheo} of Caratheodory).

Now, let $\beta_1,\beta_2$ be the components of $\gamma\setminus\left\{e^{2\pi{\bf i}\theta_{j_1}},e^{2\pi{\bf i}\theta_{j_2}}\right\}$. Then, for $i=1,2$, $\gamma_i=\alpha_{j_1}\cup \alpha\cup \alpha_{j_2}\cup\beta_i$ are both good cuts of $X$ with $\gamma_i\cap \partial X=\{z_{j_1},z_{j_2}\}$ and one of them separates $x$ from $y$.
\end{proof}

Proposition~\ref{interior-J} has a direct corollary.
\begin{coro}\label{interior-E}
If $X\subset\bbC$ is a locally connected continuum with finitely many complementary components, then the pseudo fiber $E_x$ of $X$ at a point $x\in X^o$ is trivial, {\em i.e.}, it consists of a single point.
\end{coro}

In the following, we continue to discuss basic properties of fibers and pseudo-fibers, as given in Proposition~\ref{prop:fibers}. Before that, we prove a helpful lemma.

\begin{lemm}
\label{comp-ineq}
Let $E\subset\bbC$ be a nonempty compact set such that $\piz(\bbC\setminus E)$ is finite, and $F$ a component of $E$. Then every component of $\bbC\setminus F$ contains at least one component of $\bbC\setminus E$. In particular,  $\#\piz(\bbC \setminus F) \leq \#\piz(\bbC \setminus E)$.
\end{lemm}

\begin{proof}
As $F$ is a component of $E$, it is closed in $E$. This means that there is a closed subset $B$ of $\bbC$ with $F=B\cap E$. Thus $F$ is also a nonempty compact set of $\bbC$.

Recall that every complementary component $P$ of a compact set $X$ in $\bbC$ is a path connected open set whose closure $\overline{P}$ in $\bbC$ intersects $X$. (Otherwise, $\emptyset\ne P=\overline{P}\subsetneqq\bbC$ is a clopen subset of $\bbC$, contradicting the connectedness of $\bbC$.)

Denote by $P_1,\ldots,P_n$  the components of $\bbC\setminus E$. For $1\le i\le n$, let $Q_i$ be the union of $P_i$ with all the components $M$ of $E$, other than $F$, such that $\overline{P_i}\cap M\ne\emptyset$. Then, every $Q_i$ is a connected subset of $\bbC\setminus F$. Showing that
$\displaystyle\bbC\setminus F=\bigcup_{i=1}^nQ_i$ will end our proof.

\begin{clai-nn}
Every component $M$ of $E$ intersects $\overline{P_i}$ for some $i$.
\end{clai-nn}
Suppose on the contrary that there were a component $M$ of $E$ with $M\cap\left(\bigcup_{i=1}^n\overline{P_i}\right)=\emptyset$. $M$ is a closed subset of $\mathbb{C}$, as it is closed in $E$. Since $\overline{\bbC\setminus E}=\overline{\bigcup_{i=1}^nP_i}=\bigcup_{i=1}^n\overline{P_i}$, the distance
$$\displaystyle d:=\inf\left\{|x-y|: x\in M, y\in\overline{\bbC\setminus E}\right\}
$$
between $M$ and $\overline{\bbC\setminus E}$ is positive. Fix a positive number $\epsilon$ smaller than $d$. Then the $\epsilon$-neighborhood $M_\epsilon:=\bigcup_{x\in M}\{y: |x-y|<\epsilon\}$ of $M$ is disjoint from $\overline{\bbC\setminus E}$ and hence is a subset of $E$. As $M\subsetneqq M_\epsilon\subset E$ and $M_\epsilon$ is connected, this  contradicts the fact that $M$ is a component of $E$. This proves the claim.

Using this claim, we obtain that $\displaystyle\bbC\setminus F=\bigcup_{i=1}^nQ_i$. Indeed, suppose $\displaystyle \bigcup_{i=1}^nQ_i  \subsetneq \bbC\setminus F$, and let $x\in (\bbC\setminus F)\setminus\bigcup_{i=1}^nQ_i$. Then $x\in E\setminus F$. Denoting by $M_x$ the connected component of $E$ containing $x$, we have $M_x\ne F$ and by claim $M_x\cap \overline{P_{i_0}}\ne \emptyset$ for some $i_0$. It follows that $M_x\subset Q_{i_0}$, hence $x\in Q_{i_0}$, contradicting the definition of $x$.
\end{proof}

The following proposition makes Definitions~\ref{def:order} and~\ref{def:numscale} consistent.
\begin{prop}
\label{prop:fibers}
Let $X \subseteq \bbC$ be a continuum such that $\piz(\bbC \setminus X)$ is finite.
For all $x \in X$, we have:
\begin{enumerate}
\itemsep=0pt
\item $E_x$ and $F_x$ are compact.
\item $\#\piz(\bbC \setminus F_x) \ \leq \ \#\piz(\bbC \setminus E_x) \ \leq \ \#\piz(\bbC \setminus X)$.
\end{enumerate}
\end{prop}

\begin{proof}
Let $x\in X$, $(y_n)_{n\in\mathbb{N}}$ a convergent sequence in $E_x$ and 
$y\in X$ its limit ($X$ is compact).  If $y$ can be separated from $x$ by a good cut $\gamma$, then for $k$ large enough $y_k$ is separated from $x$ by this good cut $\gamma$. This contradicts the fact that all $y_k$ belong to $E_x$.
Hence $y \in E_x$, so $E_x$ is closed and thus is compact. As $F_x$ is a component of $E_x$, we know that $F_x$ is closed in the compact set $E_x$ and hence is a closed set of $\bbC$. Consequently, it is also a compact set.\\

Remember that the complementary components of a compact set $Y$ in $\bbC$ coincide with the path-connected components of $\bbC \setminus Y$. This holds here for $Y=X,E_x$ and $F_x$.

Suppose that $\#\piz(\bbC \setminus X) = n$, \emph{i.e.}, that $\bbC \setminus X$ has $n$ components $P_1,\ldots, P_n$.  We first show that $\#\piz(\bbC \setminus E_x) \leq \#\piz(\bbC \setminus X)$. For each $i=1,\dots, n$, let $Q_i$ be the component of $\bbC\setminus E_x$ that contains $P_i$. Note that we may have $Q_i=Q_j$ for some $i\ne j$. For sure, $\bbC\setminus X\subset\bigcup_{i=1}^n Q_i\subset\bbC\setminus E_x$. We wish to prove that the latter inclusion is an equality. Let $y\in\bbC\setminus E_x$. If $y\in\bbC\setminus X$, then $y\in\bigcup_{i=1}^n Q_i $ by the above inclusion. If now $y\in X \setminus E_x$, there is a good cut $\gamma$ of $X$ that separates $x$ and $y$. Let $U_\gamma$ be the component of $\bbC \setminus \gamma$ that contains $y$. Then $\gamma$ is a good cut that separates $x$ from each point $z\in U_\gamma$, thus $U_\gamma \cap E_x$ is empty. Since the good cut $\gamma$ intersects $\bbC\setminus X=\bigcup_{i=1}^nP_i$,  $U_\gamma\cap P_i\ne\emptyset$ for some $i$ and hence $y\in U_\gamma\cup P_i$, a connected subset of $\bbC\setminus E_x$. This indicates that $y\in Q_i$. Therefore, $\bbC\setminus E_x\subset\bigcup_{i=1}^n Q_i$, hence $\bbC\setminus E_x=\bigcup_{i=1}^n Q_i$ and $\#\piz(\bbC \setminus E_x) \leq \#\piz(\bbC \setminus X)$. \\

The inequality $\#\piz(\bbC \setminus F_x) \ \leq \ \#\piz(\bbC \setminus E_x)$ follows from Lemma~\ref{comp-ineq}.
\end{proof}

The structure of fibers and pseudo-fibers is invariant by homeomorphism of the plane in the following sense.

\begin{prop}\label{Rmk:Invariance} Let $X\subset\mathbb{C}$ be a continuum such that $\piz(\bbC \setminus X)$ is finite and $h:\hat{\mathbb{C}}\to\hat{\mathbb{C}}$ a homeomorphism preserving $\infty$, \emph{i.e.}, $h(\infty)=\infty$. Let $Y:=h(X)$. Then:
\begin{itemize}
\item $\gamma$ is a good cut of $X$ separating $x,y\in X$ if and only if $h\circ\gamma$ is a good cut of $Y$ separating $h(x),h(y)$.
\item The (pseudo-)fibers of $Y$ are the images by $h$ of the (pseudo-)fibers of $X$:
$$\forall x\in X,\;E_{h(x)}=h(E_x)\;\;\textrm{and}\;\;F_{h(x)}=h(F_x).
$$
\end{itemize}
\end{prop}

\section{Trivial fibers and local connectedness}\label{sec:proofmain}
In this section, we prove our main result Theorem~\ref{main}. We prove in Theorem~\ref{trivial->lc} that if a continuum $X\subset\bbC$ is not locally connected, then it contains a non-trivial fiber. The proof  of the ``converse'' is rather intricate. We aim at showing that if a continuum $X\subset\bbC$ with $\piz(\bbC\setminus X)<\infty$ is locally connected then the pseudo-fiber $E_x$ is trivial for each $x\in X$.  We will establish the result in the case $\piz(\bbC\setminus X)=1$ (Proposition~\ref{prop:lc->trivial_base}) and then use induction (Theorem~\ref{theo:lc->trivial}).

\begin{theo}
\label{trivial->lc}
If $X$ is not locally connected at $x\in X$, then the fiber $F_x$ is nontrivial.
\end{theo}

\begin{proof}

If $X$ is not locally connected at $x\in X$
then there exists a neighborhood $V$ of $x$
such that the connected component $Q$ of $V$ containing $x$ is not a neighborhood of $x$.

Let $U$ be an open set of $\bbC$ with $(U\cap X)\subset V$ and fix a closed disk $B(x,r)$ on the plane with $B(x,r)\subset U$, by choosing small enough radius $r>0$. Then \[(B(x,r)\cap X)\subsetneqq (U\cap X)\subset V.\]
Here the component $P_x$ of $B(x,r)\cap X$ containing $x$
is not a neighborhood of $x$, since $P_x\subset Q$.
That is to say, there exist an infinite sequence of
distinct points $\{x_k\}$ with values in $(B(x,r)^o\cap X)\setminus P_x$ such that $\lim_{k\rightarrow\infty}x_k=x$.

Let $P_k$ denote the component of $B(x,r)\cap X$ that contains $x_k$ and $P_x$ the component containing $x$. As $x_k\notin P_x$ for all $k\ge1$, we may assume that $P_k\cap P_l=\emptyset$ for $k\ne l$. In the following, denote by $Q_k$ the component of $B(x,\frac{1}{2}r)\cap X$ containing $x_k$ and $Q_x$ the one containing $x$. Then $Q_x\subset P_x$ and $Q_k\subset P_k$.

By~\cite[Section 42, Chapter I, Theorem 1]{Kur68} and~\cite[Section 42, Chapter II, Theorem]{Kur68}, we may assume that $\{Q_k\}$
is a convergent sequence under Hausdorff distance, by replacing it with an appropriate subsequence. As $B(x,\frac{1}{2}r)\cap X$ is a closed proper subset of $X$,
each of its components intersects the boundary of $B(x,\frac{1}{2}r)$ \cite[Section 47, Chapter II, Theorem 1]{Kur68}. Let $Q_\infty = \lim_{k \rightarrow \infty} Q_k$, and $\diam(Q_k)$ the diameter of $Q_k$. Then $Q_\infty$ is a subcontinuum of $ Q_x$ and $\diam(Q_\infty) \geq \frac{1}{2}r$, since $x\in Q_\infty$  and $Q_\infty\cap \partial B(x,\frac{1}{2}r)\ne\emptyset$. Consequently, the proof is completed by the following claim.

\begin{clai-nn}
$Q_\infty\subset E_x$ hence $Q_\infty\subset F_x$.
\end{clai-nn}

Otherwise, there is a good cut $\gamma$ of $X$ separating $x$ from a point $y\in Q_\infty\setminus\{x\}$. Since $y \in Q_\infty$, there is a point $y_k \in Q_k$ for all $k$ such that $y\in\overline{\{y_k: k\}}$. We may assume that $\lim_{k \rightarrow \infty} y_k = y$ by replacing $\{y_k\}$ with an appropriate subsequence.
Fix an integer $N$ such that $x_k$ and $y_k$ are separated by the good cut $\gamma$ for all $k \geq N$.

Let $A_k = \gamma \cap Q_k$ for $k \geq N$. By replacing with an appropriate subsequence, we may assume that  $\{A_{k}\}$ is convergent under Hausdorff distance. Let $A_\infty = \lim_{i\rightarrow \infty} A_{k}$. We have $A_\infty \subset (Q_\infty\cap \gamma)$. Fix $u_\infty \in A_\infty$ and $u_k \in A_{k}$ with $\lim_{k \rightarrow \infty} u_k = u_\infty$.
Denote by $\gamma_1$ and $\gamma_2$ the two components of $\gamma \setminus \{u_1, u_\infty\}$.  Clearly, either $\gamma_1$ or $\gamma_2$ (say, $\gamma_1$)
contains an infinite subsequence of $\{u_k\}$.

As $\#\gamma \cap \partial X<\infty$, $\gamma_1 \setminus \partial X$ is the union of finitely many open arcs: $\gamma_1 \setminus \partial X = I_1 \cup \cdots \cup I_n$.
Then, there exists a unique $I_j$ with $u_\infty \in \overline{I_j}$. Therefore, $I_j$ contains an infinite subsequence $\{u_{k_i}\}$ of $\{u_k\}$.

On the one hand, $u_{k_i} \in X^o$ and $I_j\cap(\partial X)=\emptyset$, so we have $I_j \subseteq X^o$. On the other hand, $u_\infty\in (B(x,\frac{1}{2}r)\cap X)$, so we may fix a subarc $I$ of $I_j$ with \[u_\infty\in\overline{I}\subset(B(x,r)\cap X)\subset V.\]
Then $Q_\infty\cup I$ is a connected subset of $P_x$. However, $I$, and hence $Q_\infty\cup I$, contains infinitely many points in $\{u_{k_i}\}$, where $u_{k_i}\in P_{k_i}$ for each $i$. This contradicts the fact $P_x$ is the component of $B(x,r)\cap X$ containing $x$.
\end{proof}

We now turn to the converse part of our main theorem. Let us consider a special case.

\begin{prop}
\label{prop:lc->trivial_base}
Let $X \subseteq \bbC$ be a continuum such that $\bbC\setminus X$ is connected.
If $X$ is locally connected, then the pseudo fiber $E_x$ is trivial for all $x \in X$.
\end{prop}

\begin{proof}

By Corollary~\ref{interior-E}, the pseudo fiber $E_x$ of $X$ at $x\in X^o$ is trivial. Therefore, let $x\in \partial X$. To obtain that the pseudo fiber $E_x$ of $X$ at $x$ is trivial, we consider $y\in X\setminus\{x\}$ and show that $x$ and $y$ can be separated by a good cut of $X$.  Again, if $y\in X^o$, then the pseudo-fiber $E_y$ of $y$ is trivial, hence $x$ and $y$ can be separated by a good cut of $X$. Thus we suppose that $y\in\partial X$.

By assumption, $U_\infty:=\bbC\setminus X$ is a simply connected domain. We denote by $\varphi$ the Riemann mapping from $\{z\in\bbC: |z|>1\}$ onto $\bbC\setminus X$ that fixes $\infty$. By the theorem of Caratheodory (see Theorem~\ref{CaraTheo}), $\varphi$ can be continuously extended to the unit circle $\mathbb{S}^1:=\{z: |z|=1\}$. That is,  we consider
$$\varphi:\{z: |z|\geq1\}\to\overline{U_\infty}=\overline{\mathbb{C}\setminus X}$$ as a continuous onto mapping whose restriction to $\{z: |z|>1\}$ is a conformal mapping onto $\mathbb{C}\setminus X$.

Since the pre-images $\varphi^{-1}(x)$ and $\varphi^{-1}(y)$
are two nonempty disjoint compact sets of $\mathbb{S}^1$,
 we can find two open arcs $\alpha_1, \alpha_2$ on $\mathbb{S}^1$ with end points $a_i$ and $b_i$
such that $\varphi(a_i)=x$,
$\varphi(b_i)=y$,
and $\varphi(\alpha_i) \cap \{x,y\} = \varnothing$. Note that $\varphi$ may not be injective on the unit circle. But $\varphi(\alpha_i)\cup \{x,y\}$ is arcwise-connected for $i=1,2$, since it is the continuous image of the arc $\alpha_i\cup\{a_i,b_i\}$. Hence one can find simple arcs $A_1,A_2$ with end points $x,y$ and $A_i\subset\varphi(\alpha_i)$.

\paragraph{Case 1.} $A_1\cap A_2=\varnothing$, see Figure~\ref{fig:ray2}.

In this case, $\Gamma := A_1\cup A_2\cup\{x,y\}$
is a simple closed curve on $\partial X$. As $\mathbb{C}\setminus X$ is assumed to be connected, the bounded component $B$ of
$\mathbb{C}\setminus\Gamma$ does not intersect $\partial X$. Indeed, $B$ being an open set, from $B\cap\partial X\ne\emptyset$ would follow that $B\cap (\mathbb{C}\setminus X)\ne\emptyset$, that is, $B\cap U_\infty\ne\emptyset$. As  $\partial B=\Gamma\subset X$, a connectedness argument would lead to $U_\infty\subset B$, a contradiction.  Now, we fix a point $t_i \in \alpha_i$ such that $\varphi(t_i)\in A_i$ for $i=1,2$ and let $W_i := \{kt_i : k \in \left[1,\infty\right[\}$.
Moreover, choose an open arc $\eta$ in the bounded component of $\bbC\setminus\Gamma$ that connects $\varphi(t_1)$ and $\varphi(t_2)$.
Then $\gamma:=\varphi(W_1\cup W_2 \cup \{\infty\}) \cup \eta$
is a simple closed curve satisfying $\gamma\cap\partial X=\{\varphi(t_1),\varphi(t_2)\}$, a finite set. Finally, note that $x$ and $y$ belong to distinct components of $\mathbb{C}\setminus\gamma$. Hence $\gamma$ is a good cut of $X$ separating $x$ and $y$.
\begin{figure}[H]
\centering
\myvcenter{\includegraphics{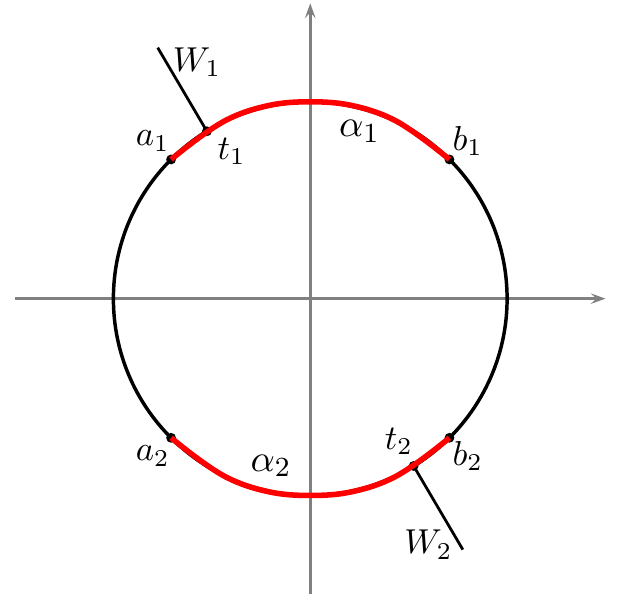}}
\myvcenter{\includegraphics{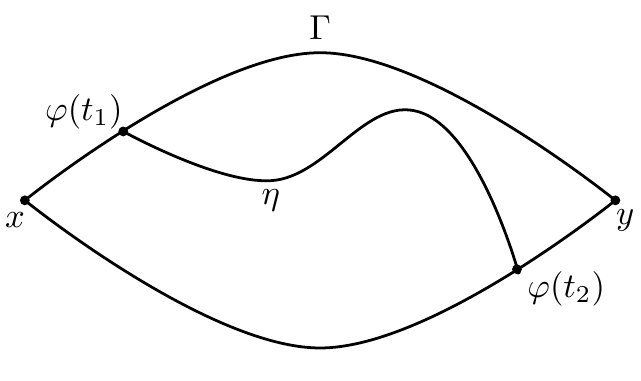}}
\caption{
Proof of Case 1 of Claim 2.}
\label{fig:ray2}
\end{figure}

\paragraph{Case 2.} $\varphi(\alpha_1)\cap\varphi(\alpha_2)\neq\varnothing$, see Figure~\ref{fig:ray3}.

In this case, there exist $t_3\in\alpha_1$ and $t_4\in\alpha_2$
with $\varphi(t_3)=\varphi(t_4)\notin\{x,y\}$.
For small $\varepsilon>0$,
choose two segments $A = \{ka_1 : 1\leq k\leq 1+\varepsilon\}$,
$B=\{kb_1: 1 \leq k \leq 1+\varepsilon\}$
and a circular arc $\beta := \{(1+\varepsilon)z : z\in\alpha_1\}$.
Moreover, let $W_i:=\{kt_i : k\in\left[1,\infty\right[\}$ for $i=3,4$.
Then $\gamma':\varphi(W_3\cup W_4\cup\{\infty\})$
is a simple closed curve whose intersection with  $\partial X$ is the finite set $\{\varphi(t_3)=\varphi(t_4)\}$.  Also, $\gamma'$ transversally crosses
the arc $\varphi(A\cup \beta\cup B)$ and hence is a good cut separating $x$ and $y$. \begin{figure}[H]
\centering
\myvcenter{\includegraphics{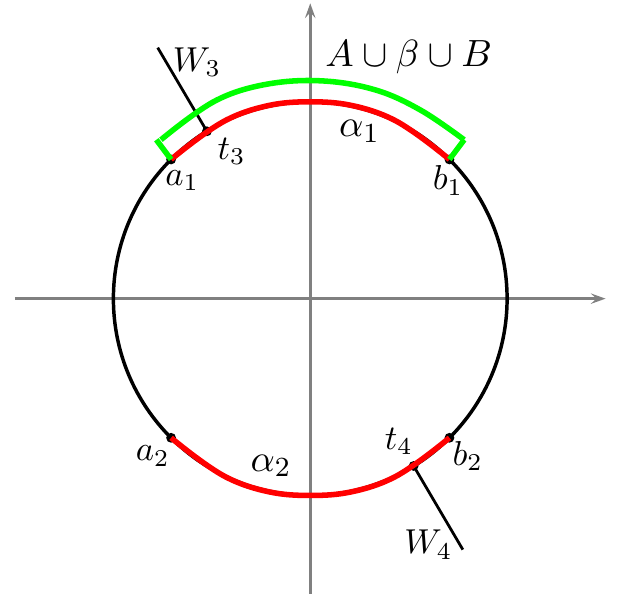}}
\myvcenter{\includegraphics{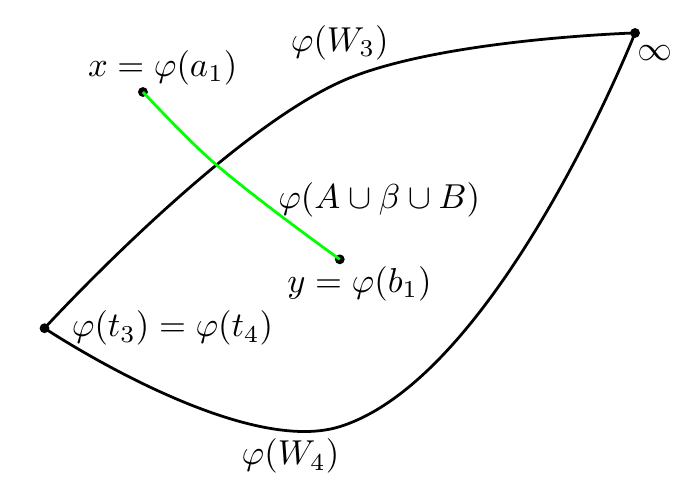}}
\caption{
Proof of Case 2 of Claim 2.
}
\label{fig:ray3}
\end{figure}
\end{proof}
\begin{rema}
Note that in Case 2 of the above proof we have been able to construct a good cut separating $x$ and $y$ that touches $X$ in only one point. This is possible when $x$ and $y$ are on ``distinct sides'' of a cut point.
\end{rema}

We finally deal with the general case.
\begin{theo}
\label{theo:lc->trivial}
Let $X \subseteq \bbC$ be a continuum such that $\pi_0(\bbC \setminus X)$ is finite.
If $X$ is locally connected, then the pseudo fiber $E_x$ is trivial for all $x \in X$.
\end{theo}
\begin{proof}
We give only a sketch of the proof. The details can be found in Section~\ref{sec:DetailedProof}.

Let $n\geq 1$ and $X$ be a locally connected continuum with $n$ complementary components. We will prove the result by induction on $n$. If $n=1$, the result holds by Proposition~\ref{prop:lc->trivial_base}. We now fix $n\geq 2$ and assume that the result holds for every locally connected continuum with at most $n-1$ complementary components.

Let $X\subseteq\mathbb{C}$ be a locally connected continuum such that  $\#\pi_0(\bbC \setminus X)=n\geq 2$. We denote by $U_1, \ldots, U_n$  the components of $\bbC \setminus X$. For  $x\in X$,  we need to show that every point $y$ in $X\setminus\{x\}$ is separated from $x$ by a good cut of $X$. By  Proposition~\ref{interior-J}, we may assume that $x\ne y\in\partial X=\bigcup_{j=1}^n\partial U_j$. The idea is to construct a mapping that is almost a homeomorphism but erases at least one complementary component in order to use the induction hypothesis. We consider two cases.

\textbf{Case 1:} $\partial U_i \ \cap \ \partial U_j = \varnothing$ for all $1\le i< j\le n$.

We choose two open disks in distinct components of $\bbC\setminus X$ and send the extended complex plane by a homeomorphism $h$ onto a cylinder $S^2$, whose top and bottom are the images of the disks, thus lie away from $h(X)$. Suppose $x\in\partial U_i$ and $y\in \partial U_j$ for some $1\le i, j\le n$. We tile the cylinder with tiny enough curvilinear rectangles on the side surface and the top and bottom disks (see Figure~\ref{fig:brickwall}). For each $k=1,\ldots n$, the union of the tiles intersecting $h(\overline{U_k})$ is a locally connected continua without cut points that covers $h(\overline{U_k})$.

If $i\ne j$, we apply Theorem \ref{TorTheo} of Torhorst to separate $h(x)$ from $h(y)$ by a simple closed curve contained in $h(X^o)$. Its pre-image lies in $X^o$ and separates $x$ from $y$, hence, by Proposition~\ref{interior-J}, we are done.

If $i=j$, a theorem of Brown~\cite[Theorem 1]{Brown} allows us to use a continuous mapping of the cylinder onto itself that shrinks down ``almost homeomorphically'' to a single point at least one complementary component $h(U_j)$ ($j\ne i$) of $h(X)$. In this way, the number of complementary components of the locally connected continuum $q(h(X))$ has been reduced to at most $n-1$, and we apply the induction hypothesis as well as Proposition~\ref{Rmk:Invariance} to get a good cut of $X$ separating $x$ and $y$.

\textbf{Case 2:} $\partial U_i \ \cap \ \partial U_j \neq \varnothing$ for some $i \neq j$.

Select such a pair $(i,j)$. Here the idea is to reduce the number of complementary components by ``gluing together'' $U_i$ and $U_j$ via an open arc. We take $x_0\in\partial U_i \ \cap \ \partial U_j $ and select two bounded Jordan domains entirely lying inside $U_i$ and $U_j$, up to their unique intersection point $x_0$ (Figure~\ref{fig:touchjor}).

We want to mimic this configuration by two right isosceles triangles $\Delta_1$, with vertices $x_0,x_0-1$ and $x_0-1+ {\bf i}$, and $\triangle_2$, with vertices $x_0,x_0+1$ and $x_0+1+{\bf i}$. To this effect, we choose a homeomorphism $h$ of the extended complex plane sending the two triangles intersecting at $x_0$ onto the bounded Jordan regions (Figure~\ref{fig:h-mapping}). Now, we glue the triangles at $x_0$ by constructing a mapping $g:\hat\bbC\to\hat\bbC$ that sends the upright segment $[x_0,x_0+{\bf i}]$ down to $x_0$ and is injective otherwise (Figure~\ref{fig:g-mapping}).

Setting $q:=h\circ g:\hat\bbC\to\hat\bbC$ and $Y:=q^{-1}(X)\setminus\{x_0+t{\bf i}: 0<t<1\}$, we prove that $Y$ is now a locally connected continuum whose complement has at most $n-1$ components and are able to apply the induction hypothesis on $Y$. However, if $\gamma$ is a good cut of $Y$ separating $x$ from $y$, $q(\gamma)$ is a locally connected continuum that may not be a simple closed curve and has finite, but possibly empty intersection with $\partial X$. Therefore, we distinguish two subcases. 

If $x\ne x_0\ne y$, using Theorem~\ref{TorTheo}  of Torhorst, we find a simple closed curve $J$ on the boundary of $\hat\bbC\setminus q(\gamma)$ separating $x$ and $y$. In the case $J$ is not a good cut, that is, $J\subset X$, we apply the theorem of Brown~\cite[Theorem 1]{Brown} to shrink $J$ together with its complementary component containing $y$ ``almost homeomorphically'' to a single point, via a mapping $\varphi$. We then apply again the induction hypothesis, on $\varphi(X)$, and obtain a good cut of $X$.

Otherwise, we consider \emph{w.l.o.g.} $x_0=x$. We enclose the segment $[x_0,x_0+{\bf i}]$ by a locally connected continuum with no cut points and avoiding $q^{-1}(y)$: we use here two good cuts of $Y$ that separate $x_0$ and $x_0+{\bf i}$ from $q^{-1}(y)$ as well as a simple closed curve disjoint from $Y$ and connecting the good cuts. By Theorem \ref{TorTheo} of Torhorst, we obtain a simple closed curve $K^ {**}$ whose intersection with $\partial Y$ is finite, possibly empty. The proof is then similar as in the above subcase.
\end{proof}


\section{Details for the proof of Theorem~\ref{theo:lc->trivial}}\label{sec:DetailedProof}

This section is entirely devoted to the proof of Theorem~\ref{theo:lc->trivial}, sketched at the end of the preceding section.  Let  $n\geq 1$ and $X$ be a locally connected continuum with $n$ complementary components. We prove the result by induction on $n$. For $n=1$, the result holds thanks to Proposition~\ref{prop:lc->trivial_base}. Let $n\geq 2$ and assume that the result holds for every locally connected continuum having at most $n-1$ complementary components.

Let $X\subseteq\mathbb{C}$ be a locally connected continuum such that  $\#\pi_0(\bbC \setminus X)=n\geq 2$. Let $U_1, \ldots, U_n$ denote the components of $\bbC \setminus X$. By Corollary~\ref{interior-E}, we just need to show that every two points  $x\ne y\in\partial X=\bigcup_{j=1}^n\partial U_j$ can be separated from $x$ by a good cut of $X$.

\subsection{Case 1: $\partial U_i \ \cap \ \partial U_j = \varnothing$ for all $1\le i< j\le n$.}

Fix a point $a_i\in U_i$ for $i=1,2$ and a small enough number $\delta>0$ such that the circle $C_i=\{z: |z-a_i|=\delta\}$ is contained in $U_i$. Let $D_i$ be the bounded component of $\mathbb{C}\setminus C_i$. Choose two simple closed curves $J_1$ and $J_2$ as indicated in Figure \ref{fig:annulus}.
\begin{figure}[H]
\centering
\myvcenter{\includegraphics{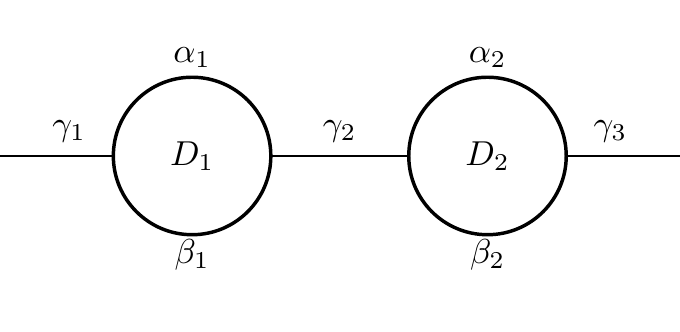}}
\caption{$J_1=\gamma_1\cup\alpha_1\cup\gamma_2\cup\alpha_2\cup\gamma_3\cup\{\infty\}$, $J_2=\gamma_1\cup\beta_1\cup\gamma_2\cup\beta_2\cup\gamma_3\cup\{\infty\}$.}
\label{fig:annulus}
\end{figure}
Let $\bbS^1=\{z\in\bbC: |z|=1\}$ denote the unit circle on $\bbC$. Denote by $W_1$ and $W_2$  the two components of $\left(\bbS^1\setminus\{\pm1\}\right)\times(0,1)$. Then $\partial W_1$ and $\partial W_2$ are simple closed curves that share the two disjoint arcs $\{\pm1\}\times[0,1]$. Therefore we may find  homeomorphisms $f_i$ for $i=1,2$ sending $J_i$ onto $\partial W_i$ such that $f_1(x)=f_2(x)$ for each $x$ in $\gamma_1\cup\gamma_2\cup\gamma_3\cup\{\infty\}$.
By Sch\"onflies theorem we may find a homeomorphism $h$ from $\hat{\bbC}$ onto the topological sphere
\[S^2=\left(\bbD^1\times\{0\}\right)\cup\left(\bbD^1\times\{1\}\right) \cup\left(\bbS^1\times[0,1]\right),\]
such that $h(x)=f_i(x)$ for $x\in J_i$. Here $\bbD^1=\{z\in\bbC: |z|\le 1\}$ is the unit disk on $\bbC$. Clearly, we have $h(C_i)=\bbS^1\times\{i-1\}$ for $i=1,2$ and $h\left(\hat{\bbC}\setminus(D_1\cup D_2)\right)=\bbS^1\times[0,1]$.

Consider $S^2$ as a subset of $\bbC\times\bbR$. Let $\Delta=\min\left\{ \textrm{dist}\left(h\left(\partial U_i\right), h\left(\partial U_j\right)\right): 1\le i< j\le n\right\}$, where $\textrm{dist}(A,B)=\inf\{\|x-y\|: x\in A, y\in B\}$ gives the distance between subsets $A, B$ of $\bbC\times\bbR$ under Euclidean metric $\|\cdot\|$. Choose a large enough number $L$ with $\sqrt{\left|e^{2\pi{\bf i}/L}-1\right|^2+\frac{1}{L^2}}<\frac{\Delta}{2}$.
Then $\mathcal{T}=\{\bbD^1\times\{0\}\}\cup\{\bbD^1\times\{1\}\}\cup\left\{T_{j,k}: 1\le j,k\le L\right\}$ is a tiling of $S^2$, where
\[T_{j,k}=\left\{e^{2t\pi{\bf i}/L}: \frac{j-1}{L}+\frac{1+(-1)^k}{2}\le t\le\frac{j}{L}+\frac{1+(-1)^k}{2}\right\}\times\left[\frac{k-1}{L},\frac{k}{L}\right].\]
Please see Figure \ref{fig:brickwall} for a depiction of $T_{j,k}$ with $L=6$.
\begin{figure}[H]
\centering
\myvcenter{\includegraphics{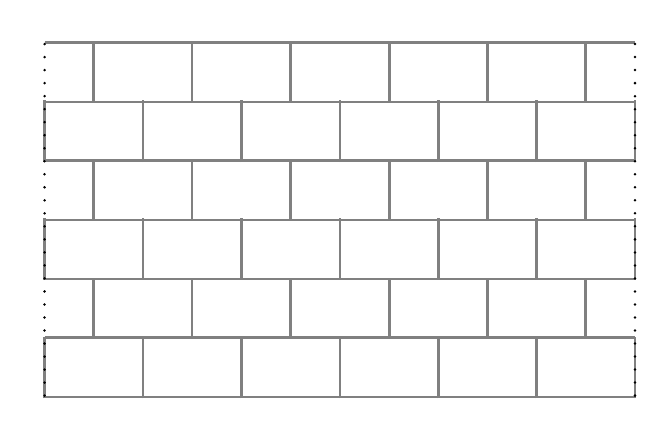}}
\caption{The dotted lines indicate how to obtain $\bbS^1\times[0,1]$ by identification.}
\label{fig:brickwall}
\end{figure}

Let $P_i$ be the union of all the tiles in $\mathcal{T}$ that intersect $h\left(\overline{U_i}\right)$ for $1\le i\le n$.  Since every tile of $\mathcal{T}$ is either a rectangle or a closed disk, and since the intersection of two tiles in $\mathcal{T}$ is either empty or a non-degenerate interval, we can make the following observations.
\begin{itemize}
\item[1.] $P_i \;(1\le i\le n)$ are pairwise disjoint locally connected continua without cut points.
\item[2.] For all $1\leq i\ne j\leq n$, $P_j$ is contained in a single component of $S^2\setminus P_i$.
\end{itemize}
For $1\le i\le n$, denote by $W_{i,t}\; ( 1\le t\le l_i)$ the components of $S^2\setminus P_i$. Then, by~Theorem~\ref{TorTheo}, the boundary of each $W_{i,t}$ is a simple closed curve contained in $h(X^o)$. Consequently, the sets $\overline{W_{i,t}} \;(1\le t\le l_i)$ are pairwise disjoint topological disks satisfying $\partial W_{i,t}\subset h(X^o)$.

If $x\in \partial U_i$ and $y\in\partial U_j$ for $i\ne j$, then $h(y)\in P_j$ belongs to a component $W_{i,t}$ of $S^2\setminus P_i$ and hence $h^{-1}(\partial W_{i,t})$ is a simple closed curve in $X^o$ that separates $x$ and $y$. By Proposition~\ref{interior-J}, we can infer that $x$ and $y$ are separated by a good cut of $X$.

If $\{x,y\}\subset\partial U_i$ for some $i$ then $\{h(x), h(y)\}\subset h(\partial U_i)\subset P_i^o$. Choose
$1\leq t_0\leq l_i$ such that $W_{i,t_0}$ contains at least one $P_j$ for some $j\ne i$. By the above observation 2., we have the following partition:
$$\displaystyle\{1,2,\ldots,n\}=\underbrace{\{1\leq k\leq n;h(\overline{U_k})\cap W_{i,t_0}=\emptyset\}}_{\displaystyle=:K}\;\cup\;
\underbrace{\{1\leq k\leq n;h(\overline{U_k})\subset W_{i,t_0}\}}_{\displaystyle =K^c\ne\emptyset}.
$$
Clearly, $i\in K$. Let $q:S^2\to S^2$ be an onto and continuous mapping sending $\overline{W_{i,t_0}}$ to a point and injective otherwise, \emph{i.e.},
$$
\begin{array}{c}
\exists z\in S^2, q(\overline{W_{i,t_0}})=\{z\},\\
q \textrm{ is injective on }S^2\setminus\overline{W_{i,t_0}}
\end{array}
$$
(see also~\cite[Theorem 1]{Brown}).  Then $z$ is in the interior of $q(h(X))$ and $q$ induces a homeomorphism from $S^2\setminus\overline{W_{i,t_0}} $ onto $S^2\setminus\{z\}$. It follows that $q(h(X))$ is a locally connected continuum, and that
$$q(h(X))^c=S^2\setminus q(h(X)) = \bigcup_{k\in K}q(h(U_k)) 
$$
is a disjoint union containing $q(h(U_i))$. Indeed, $q(h(\overline{U_k}))=\{z\}\subset q(h(X))$ for all $k\notin K$. Therefore, $q(h(X))$ has at most $n-1$ complementary components. By induction hypothesis, all the fibers of $q(h(X))$ are trivial. In particular, there is a good cut $\Gamma$ of $q(h(X))$ separating $q(h(x))$ and $q(h(y))$. We may choose $\Gamma$ in a way that it avoids the point $z$, which lies in the interior of $q(h(X))$.
Recall that $h: \hat{\bbC}\rightarrow S^2$ is a homeomorphism and that $\left.q\right|_ {S^2\setminus\overline{W_{i,t_0}}}$ is a homeomorphism from $S^2\setminus\overline{W_{i,t_0}} $ onto $S^2\setminus\{z\}$.
Therefore, $q^{-1}(\Gamma)$ is a good cut of $h(X)$ separating $h(x)$ and $h(y)$, hence that $h^{-1}(q^{-1}(\Gamma))$ is a good cut of $X$ separating $x$ and $y$ (see Proposition~\ref{Rmk:Invariance}).

\subsection{Case 2: $\partial U_i \ \cap \ \partial U_j \neq \varnothing$ for some $i \neq j$.}

Fix a point $x_0\in\partial U_i\cap \partial U_j$. By the Riemann mapping theorem, there are two conformal mappings
$$\phi_i:\mathbb{D}\to U_i,\;\;\phi_j:\mathbb{D}\to U_j,
$$
where $\mathbb{D}:=\{z\in\mathbb{C}:|z|<1\}$. Since $X$ is locally connected, $\phi_i$ and $\phi_j$ extend continuously to the closure $\mathbb{D}^1$ of $\mathbb{D}$ (see Theorem~\ref{CaraTheo} of Caratheodory). We may assume that $\phi_i(1)=\phi_j(1)=x_0$. Let
$$\Gamma=\{z\in\mathbb{C}:|z-1/2|=1/2\},\; J_1 =\phi_i(\Gamma),\;J_2=\phi_j(\Gamma).
$$
Then $J_1\cap J_2=\{x_0\}$, $J_1\setminus\{x_0\}\subset U_i$ and $J_2\setminus\{x_0\}\subset U_j$, as indicated in Figure~\ref{fig:touchjor}.

\begin{figure}[H]
\centering
\myvcenter{\includegraphics{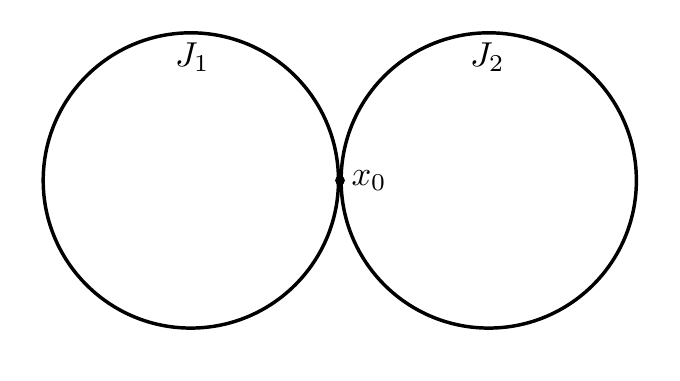}}
\caption{For the sake of convenience, we represent $J_1,J_2$ as two circles.}
\label{fig:touchjor}
\end{figure}

Let $\Phi$ be a conformal mapping from $\{z\in\mathbb{C}:|z|>1\}$ onto the unbounded component of $\mathbb{C}\setminus (J_1\cup J_2)$. As $J_1\cup J_2$ is a locally connected continuum, $\Phi$ can be continuously extended to $\{z\in\mathbb{C}:|z|\geq 1\}$ (Theorem~\ref{CaraTheo} of Caratheodory). Therefore, one can find two points $z_1,z_2$ on the unit circle satisfying $\Phi(z_1)\in U_i$ and $\Phi(z_2)\in U_j$. Put $R_i:=\{t z_i:t\geq1\} \;\;(i=1,2)$.
Then
$$\{\infty\}\cup \Phi (R_1\cup R_2 )\cup J_1\cup J_2
$$
gives rise to a partition of $\hat{\bbC}$ into four Jordan domains, \emph{i.e.}, where each domain is bounded by a simple closed curve. See the left part of Figure~\ref{fig:h-mapping}.

\begin{figure}[H]
\centering
\myvcenter{\includegraphics{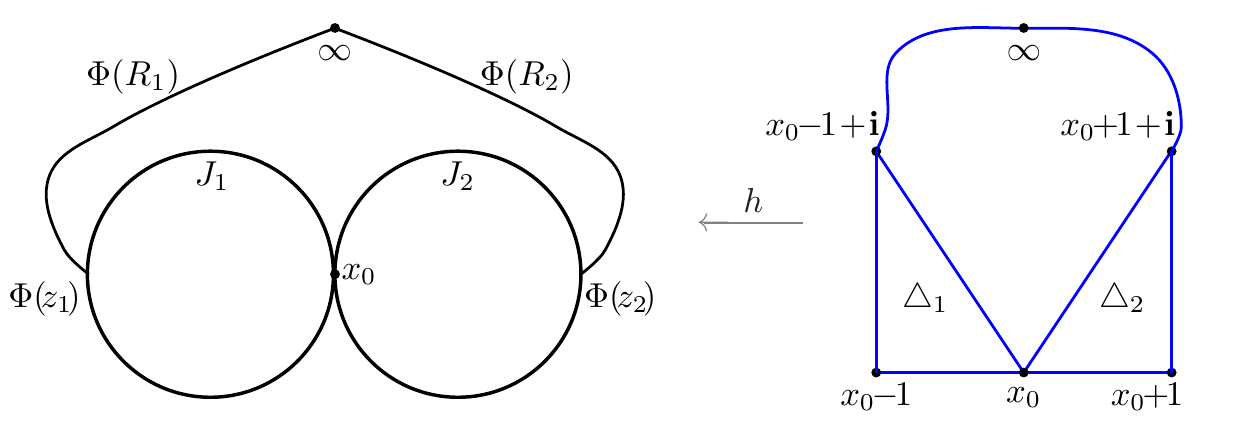}}
\caption{Mapping $h$ (we also marked $\infty$ in the figure).}
\label{fig:h-mapping}
\end{figure}

By Sch\"onflies Theorem, we can choose a homeomorphism $h:\hat\bbC\to \hat\bbC$ such that
$$h(\infty)=\infty,\;h(x_0)=x_0,\;h(\partial\Delta_i)=J_i\;\;(i=1,2).
$$
Here, $\Delta_1$ is the triangle with vertices $x_0,x_0-1$ and $x_0-1+ {\bf i}$   and $\triangle_2$ the triangle with vertices $x_0,x_0+1$ and $x_0+1+{\bf i}$. See Figure \ref{fig:h-mapping}. 

Let $W_1$ be the half plane to the left of the vertical line $L_1$ through $x_0-1$ and $W_2$ the one to the right of the vertical line $L_2$ through $x_0+1$. Let $W_3$ be the half plane below the horizontal line through $x_0$. See Figure~\ref{fig:g-mapping} for relative locations of $x_0, L_1, L_2$.
\begin{figure}[H]
\centering
\myvcenter{\includegraphics{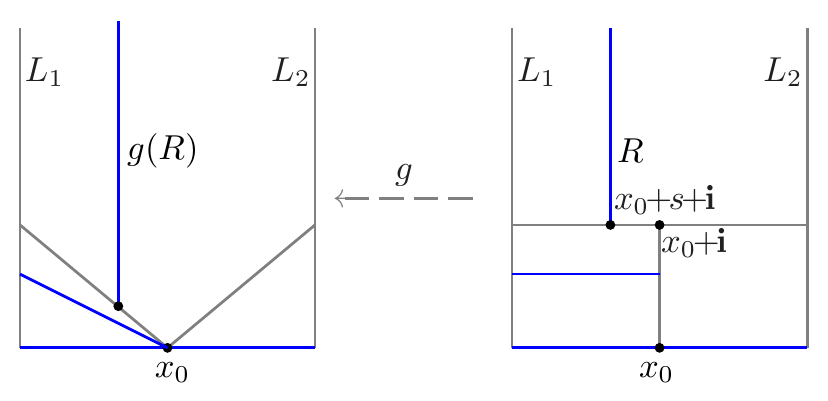}}
\caption{Mapping $g$.}
\label{fig:g-mapping}
\end{figure}
Then, we further define a continuous onto mapping $g:\hat\bbC\rightarrow\hat\bbC$ as follows (see also Figure~\ref{fig:g-mapping}):
\begin{itemize}
\item[(1)] $g(z):=z$ for $z$ in $\overline{W_1\cup W_2\cup W_3}$;
\item[(2)] $g(x_0+t{\bf i}):=x_0$ for $t\in(0,1]$, and $g$ is linearly extended to each horizontal segment between a point $x_0+t{\bf i}$ and a point $x_0\pm1+t{\bf i}\in (a\cup b)$ for $t\in(0,1]$;
\item[(3)]  $g\left(x_0+s+\left(1+t\right){\bf i}\right):=g\left(x_0+s+{\bf i}\right)+t{\bf i}$ for $s\in(-1,1)$ and $t\ge0$. See Figure \ref{fig:g-mapping} for relative locations of a vertical ray $R$ and its image $g(R)$.
\item[(4)] $g(\infty)=\infty$.
\end{itemize}

Consequently, $q=h\!\circ\! g: \hat\bbC\rightarrow\hat\bbC$ is a continuous onto mapping with
$$q^{-1}(x_0)=\{x_0+t{\bf i}: 0\le t\le1\},
$$ whose restriction $\left.q\right|_{\bbC\setminus\{x_0+t{\bf i}:\ 0\le t\le1\}}$ is a homeomorphism onto $\bbC\setminus\{x_0\}$. Let us write
  $$Y:=q^{-1}(X)\setminus\{x_0+t{\bf i}: 0<t<1\}.
  $$

Then we have
$$Y =\overline{q^{-1}(X\setminus\{x_0\})}$$
 and $Y$ a locally connected continuum. We refer to the Appendix in Section~\ref{sec:Appendix} for a complete proof of this assertion.

Finally, $Y$ has at most $n-1$ complementary components because \[q^{-1}(U_i)\cup q^{-1}(U_j)\cup\left\{x_0+s+\frac{1}{2}{\bf i}:\ -1\le s\le 1\right\}\] is a connected set disjoint from $Y$.

\subsubsection{Subcase 2.1: $x_0 \notin\{ x, y\}$.}

The induction hypothesis implies that the pseudo fiber of $Y$ at $q^{-1}(x)$ consists of a single point $q^{-1}(x)$. So, we may choose a good cut $\gamma$ of $Y$ separating $q^{-1}(x)$ from $q^{-1}(y)$. By our choice of $q$ and the definition of ``good cut'', we know that $q(\gamma)$ is a locally connected continuum and that $q(\gamma) \cap \partial X$ is a finite set (possibly empty).

Let $U_x$ denote the component of $\hat{\bbC}\setminus q(\gamma)$ containing $x$ and $U_y$ the one containing $y$. Clearly, $(\partial U_x\cup\partial U_y)\subset q(\gamma)$.
Since the restriction of $q$ to $\bbC\setminus\{x_0+t{\bf i}:\ 0\le t\le1\}$ is a homeomorphism onto $\bbC\setminus\{x_0\}$, we have $U_x\cap U_y=\emptyset$. Otherwise, we would have $U_x=U_y$, implying the existence of an arc $\alpha\subset U_x$ with end points $x$ and $y$. Then $q^{-1}(\alpha)$, a continuum in $\bbC\setminus\gamma$, would contain both $q^{-1}(x)$ and $q^{-1}(y)$, contradicting the fact that $\gamma$ is a good cut of $Y$ separating $q^{-1}(x)$ and $q^{-1}(y)$.

Since $q(\gamma)$ is locally connected, we infer that $\overline{U_x}\subset\hat{\bbC}$ is a locally connected continuum with no cut points. By Torhorst theorem (Theorem~\ref{TorTheo}), the component of $\hat{\bbC}\setminus\overline{U_x}$ containing $y$ is bounded by a simple closed curve $J$. As $J\subset\partial U_x\subset q(\gamma)$, the intersection $J\cap\partial X$ is a finite set (possibly empty).

If $J\setminus X\ne\emptyset$ then $J$ is already a good cut of $X$ separating $x$ and $y$.

If $J\subset X$, the component $W$ of $\hat{\bbC}\setminus J$ containing $y$  intersects and hence contains a component of $\hat{\bbC}\setminus X$. In fact, if $V$ denotes the other component of $\hat{\bbC}\setminus J$, then every complementary component of $X$ is entirely contained either in $V$ or in $W$. Let us choose a continuous onto map $\varphi:\hat{\bbC}\rightarrow\hat{\bbC}$ sending $J\cup W$ to a single point $z$ and injective otherwise~\cite{Brown}. Then, $\varphi(X)$ is a locally connected continuum whose complement has at most $n-1$ components. Indeed,
$$\#\pi_0(\bbC\setminus\varphi(X))=\#\left\{P\in\pi_0(\bbC\setminus X):P\subset V\right\}\leq n-1,
$$
since at least one complementary component of $X$ is contained in $W$.

Consequently, by induction hypothesis,
there exists a good cut $J^*$ of $\varphi(X)$ separating $\varphi(x)$ and $\varphi(y)=z$. As $\left.\varphi\right|_{\hat{\bbC}\setminus(J\cup W)}$ is a homeomorphism from $\hat{\bbC}\setminus(J\cup W)$ onto $\hat{\bbC}\setminus\{\varphi(y)\}$ and $J^*\cap\{\varphi(x), \varphi(y)\}=\emptyset$, we know that $\varphi^{-1}(J^*)$ is a good cut of $X$ separating $x$ from $y$.

\subsubsection{Subcase 2.2: $x_0 = x$ or $x_0=y$.}

We only consider the case $x_0=x$.

By induction hypothesis, the pseudo-fibers of $Y$ at $x_0$ and $x_0+{\bf i}$ both consist of a single point. So, we may choose good cuts $K_1$ and $K_2$ of $Y$ that respectively separates $x_0$ and $x_0+{\bf i}$ from $q^{-1}(y)$. Denote by $V_1$ the component of $\hat{\bbC}\setminus K_1$ containing $x_0$ and by $V_2$ the component of $\hat{\bbC}\setminus K_2$ containing $x_0+{\bf i}$.

Let $y_1\in\{x_0+t{\bf i}: 0<t<1\}$ be the first point from $x_0$ to $x_0+{\bf i}$ such that $y_1\in K_1$; let  $y_2\in\{x_0+t{\bf i}: 0<t<1\}$ be the first point from $x_0+{\bf i}$ to $x_0$ such that $y_2\in K_2$. Then the segment $\beta$ between $y_1$ and $y_2$ is disjoint from $Y$. Therefore, we can choose a simple closed curve $K_3$ disjoint from $Y$ such that the bounded component of $\hat{\bbC}\setminus K_3$, denoted as $V_3$, contains $\beta$ and satisfies $V_3\cap Y=\emptyset$.

Consequently, $U=V_1\cup V_2\cup V_3$ is a connected open set whose closure in $\hat{\bbC}$ is the union of three topological disks. Clearly, $\overline{U}$ is a locally connected continuum with no cut points, and $\partial U$ is a subset of $K_1\cup K_2\cup K_3$.

Since  $q^{-1}(y)\notin\overline{U}$, by Torhorst theorem (Theorem~\ref{TorTheo}), the component of $\hat{\bbC}\setminus\overline{U}$ containing $q^{-1}(y)$ is bounded by a simple closed curve $K^{*}$, which is necessarily a subset of $\partial U\subset(K_1\cup K_2\cup K_3)$ and hence intersects $\partial Y$ at most at finitely many points. Moreover, $K^{*}$ is disjoint from $\{x_0,x_0+{\bf i}\}$.

If $K^{*}\setminus Y\ne\emptyset$ then $q(K^{*})$ is already a good cut of $X$ that separates $x$ from $y$.

Otherwise, $K^{*}\subset (Y\setminus\{x_0,x_0+{\bf i}\})$ and hence $q(K^{*})$ is a simple closed curve contained in $X$ such that $x$ and $y$ belong to distinct component of $\hat{\bbC}\setminus q(K^{*})$.

In the latter case, let $V_x$ be the component of $\hat{\bbC}\setminus q(K^{*})$ containing $x$. Let $\psi: \hat{\bbC}\rightarrow\hat{\bbC}$  be a continuous onto mapping sending $q(K^{*})\cup V_x$ to a single point and injective otherwise~\cite{Brown}. Since $x_0=x\in\partial U_i\cap \partial U_j$, $V_x$ contains the connected set $U_i\cup \{x_0\}\cup U_j$ hence $\bbC\setminus\psi(X)$ has at most $n-1$ components. By induction hypothesis, there exists a good cut $\eta$ of $\psi(X)$ separating $\psi(x)$ and $\psi(y)$. As $\left.\psi\right|_{\hat{\bbC}\setminus(K^{*}\cup V_x)}$ is a homeomorphism from $\hat{\bbC}\setminus(K^{*}\cup V_x)$ onto $\hat{\bbC}\setminus\{\psi(x)\}$ and $\eta\cap\{\psi(x), \psi(y)\}=\emptyset$, we know that $\psi^{-1}(\eta)$ is a good cut of $X$ separating $x$ from $y$.


\section{Further questions}\label{sec:questions}

This section poses some questions concerning the estimate of $\ell(X)$ for particular continua $X$ on the plane. The beginning questions ask about continua $X$ with specific properties.

\begin{ques}
For every integer $k$, find in the literature of continuum theory a (path-connected) continuum $X$ such that $\ell(X) = k$.
\end{ques}

\begin{ques}
Construct a continuum $X$ which does not contain any indecomposable subcontinuum but for which $\ell(X) = \infty$.
\end{ques}

\begin{ques}
Does there exist a Julia set $J$ such that $\ell(J) = k$ for some integer $k>0$ ? (Note: some Julia sets are indecomposable and hence do not verify this.)
\end{ques}

The famous MLC conjecture says that {\em the Mandelbrot set ${\mathcal M}$ is locally connected}. Equivalently, $\ell({\mathcal M})=0$. There have been many authors obtaining local connectedness of ${\mathcal M}$ at more and more points, while the conjecture remains open. Here, we may use our scale to pose ``weaker version(s)'' of MLC as follows.
\begin{ques}
Is $\ell({\mathcal M}) \leq k$ for some integer $k$ ? ($k$-MLC)
\end{ques}

In particular, we want to verify whether $\ell({\mathcal M})\le1$. Of course, the first step toward this direction shall be discussions on the structure of fibers of ${\mathcal M}$, followed by the search of concrete condition(s) which imply that $\ell({\mathcal M})\le 1$.

Another question we are interested in is:

\begin{ques}
What are the links between $1$-MLC and the stability conjecture \cite[p.3, Conjecture 1.2]{McM-94} ?
\end{ques}



\section{Appendix}\label{sec:Appendix}
In this appendix, we recall some definitions and obtain basic results of purely topological nature, used in Section~\ref{sec:DetailedProof}.

\begin{defi} We say that a topological $X$ is \emph{locally connected at} $x$ if for every open set $U$ containing $x$ there exists a connected, open set $V$ with $x\in V\subset U$. The space $X$ is said to be \emph{locally connected} if it is locally connected at $x$ for all $x \in X$.
\end{defi}

\begin{defi} \label{def:qc}For a point $x$ of a topological space $X$, the \emph{quasicomponent} of $x$ is the set of points $y\in X$ such that there exists no separation $X=A\cup B$  into two disjoint open sets $A,B$ such that $x\in X$ and $y\in Y$ (see~\cite[Section 46, Chapter V]{Kur68}). 
\end{defi}
\begin{rema}\label{rem:qc}In a compact space, the quasicomponents are connected and coincide therefore with  the components (see~\cite[Section 47, Chapter II, Theorem 2]{Kur68}). \\
\end{rema}
In the rest of the appendix, we fix a point $x_0\in\hat{\bbC}$ and a continuous onto mapping  $g:\hat{\bbC}\to\hat{\bbC}$ such that $g^{-1}(x_0)$ is the vertical line segment $L$ between $x_0$ and $x_0+{\bf i}$, while $g^{-1}(x)$ a single point set for each $x\ne x_0$. We denote by $L^o$ the interior of $L$. 
\begin{prop}\label{A1}
If $X$ is a continuum, so is $g^{-1}(X)$.
\end{prop}
\begin{proof}
As $g$ is a continuous surjection, we know that $g^{-1}(X)$ is a compact set whose image under $g$ is exactly $X$. Suppose that $g^{-1}(X)$ is disconnected, we may fix a separation $g^{-1}(X)=A\cup B$, where $A$ and $B$ are disjoint compact nonempty sets in $\hat{\bbC}$. Then, for any $x\in X$, $g^{-1}(x)$ (a single point or the segment $L$) is contained either entirely in $A$ or entirely in $B$. It follows that $g(A)$ and $g(B)$ are disjoint nonempty sets in $\hat{\bbC}$. This forms a separation $X=g(A)\cup g(B)$, contradicting the connectedness of $X$.
\end{proof}

\begin{prop}\label{A2}
Let $X$ be a continuum with $x_0\in X$ and $\overline{g^{-1}(X)\setminus L}\;\cap \;L=\{x_0,x_0+{\bf i}\}$. Then 
$$g^{-1}(X\setminus \{x_0\})=g^{-1}(X)\setminus L^o=\overline{g^{-1}(X)\setminus L},
$$ 
and this compact set has at most two components.
\end{prop}
\begin{proof} The first equality is trivial.  Moreover, the assumption $\overline{g^{-1}(X)\setminus L}\;\cap \;L=\{x_0,x_0+{\bf i}\}$ implies that $\overline{g^{-1}(X)\setminus L}\;\cap\; L^o=\emptyset$, while $\overline{g^{-1}(X)\setminus L}\;\cup\; L^o=g^{-1}(X)$. This proves the second equality.

By Proposition~\ref{A1}, $g^{-1}(X)$ is a continuum, hence $\overline{g^{-1}(X)\setminus L}$ is a compact set. 

We finally show that any connected component of $\overline{g^{-1}(X)\setminus L}=g^{-1}(X)\setminus L^o$ contains at least $x_0$ or $x_0+{\bf i}$. Otherwise, there exists a component $Q$ of $g^{-1}(X)\setminus L^o$ with $Q\cap L=\emptyset$. As $g^{-1}(X)\setminus L^o$ is a compact set,  by Remark~\ref{rem:qc} we may choose two separations, $g^{-1}(X)\setminus L^o=R_1\cup R_2$ and $g^{-1}(X)\setminus L^o=R_3\cup R_4$, such that $x_0\in R_1$, $x_0+{\bf i}\in R_3$, and $Q\subset R_2\cap R_4$. Since $A=g(R_1\cup R_2)$ and $B=g(R_3\cap R_4)$ are disjoint nonempty compact sets with $A\cup B=X$, we obtain a contradiction to the connectedness of $X$.
\end{proof}

\begin{rema}\label{rem:diskcover} The equality of the above proposition means that $L$ is an isolated line inside $g^{-1}(X)$: for every $x_0+t{\bf i}\in L^o$ with $0<t<1$, there exists $\epsilon>0$ such that 
$$B(x_0+t{\bf i},\epsilon)\cap g^{-1}(X)=B(x_0+t{\bf i},\epsilon)\cap L^o,
$$ 
$B(x,\epsilon)$ is the disk centered at $x$ with radius $\epsilon$. 
\end{rema}

\begin{prop}\label{A3}
Let $X$ be a locally connected continuum with $x_0\in X$ and $\overline{g^{-1}(X)\setminus L}\cap L=\{x_0,x_0+{\bf i}\}$. Then $g^{-1}(X)$ is also a locally connected continuum.
\end{prop}
\begin{proof}
By Proposition \ref{A1}, we just need to check that $g^{-1}(X)$ is locally connected. By the assumption $\overline{g^{-1}(X)\setminus L}\cap L=\{x_0,x_0+{\bf i}\}$, $g^{-1}(X)$ is locally connected at every point on $L\setminus\{x_0,x_0+{\bf i}\}$ (see Remark~\ref{rem:diskcover}). Moreover,  $g$ restricted to $g^{-1}(X)\setminus L$ is a homeomorphism, hence $g^{-1}(X)$ is locally connected at every point of $g^{-1}(X)\setminus L$. It remains to show that $g^{-1}(X)$ is locally connected at $x_0$ and $x_0+{\bf i}$.

Let $r>0$. Let $U_r$ be the closed disk of radius $r$ centered at $x_0$, and $V_r$ the closed disk of radius $r$ centered at $x_0+{\bf i}$. Let $W_r$ be the  union of closed disks of radius $\delta_r$ centered at a point on $L$ between $x_0+\frac{r}{2}{\bf i}$ and $x_0+(1-\frac{r}{2}){\bf i}$, for some small enough $\delta_r$ such that $W_r$ is disjoint from $g^{-1}(X)\setminus L$ (see again Remark~\ref{rem:diskcover}).

Then $N_r:=U_r\cup V_r\cup W_r$ is a continuum whose interior contains $L$, and hence $g(N_r)$ is a continuum containing $x_0$ in its interior. By local connectedness of $X$, the component $P_r$ of $X\cap g(N_r)$ containing $x_0$ is a neighborhood of $x_0$ in $X$. Therefore, $g^{-1}(P_r)$ is a continuum by Proposition \ref{A1}; moreover, it is a neighborhood of both $x_0$ and $x_0+{\bf i}$ in $g^{-1}(X)$.

By  Proposition \ref{A2}, we further have that $g^{-1}(P_r)\setminus L^o$ is a compact set having at most two components. Let $Q_1$ be the component  containing $x_0$, and $Q_2$ the one containing $x_0+{\bf i}$. Then, $A_r:=Q_1\cup\{x_0+t{\bf i}: 0\le t\le\frac{1}{2}r\}$ and $B_r:=Q_2\cup\{x_0+t{\bf i}: 1-\frac{1}{2}r\le t\le1\}$ are connected neighborhoods of $x_0$ and $x_0+{\bf i}$ in $g^{-1}(X)$, respectively. Since  $r$ was chosen arbitrarily, this proves that $g^{-1}(X)$ is locally connected at $x_0$ and $x_0+{\bf i}$.
\end{proof}

\begin{prop}\label{A4}
Let $X$ be a locally connected continuum with $x_0\in X$ and $\overline{g^{-1}(X)\setminus L}\cap L=\{x_0,x_0+{\bf i}\}$. Then $g^{-1}(X)\setminus L^o$ is either a locally connected continuum or the union of two locally connected continua.
\end{prop}
\begin{proof}
By Proposition \ref{A2}, the compact set $g^{-1}(X)\setminus L^o$ contains at most two components, each of which  is then a  continuum. By Proposition \ref{A3},  $g^{-1}(X)$ is a locally connected continuum. Therefore, each component of $g^{-1}(X)\setminus L^o$ is locally connected at every point $x\notin\{x_0,x_0+{\bf i}\}$.

Given a component $P$ of $g^{-1}(X)\setminus L^o$ and $x$ in $P\cap\{x_0,x_0+{\bf i}\}$, we may choose for any real number $r\in(0,\frac{1}{2})$ a connected neighborhood $U_r$ of $x$ in $g^{-1}(X)$ such that $U_r$ is contained in the open disk of radius $r$ centered at $x$. Clearly, $U_r\setminus L^o$ is a connected set, which is also a neighborhood of $x$ in $g^{-1}(X)\setminus L^o$. Since  $r$ was chosen arbitrarily, we conclude that $P$ is locally connected at $x$.
\end{proof}

\begin{prop}\label{A5}
Let $X$ be a locally connected continuum with $x_0\in X$ and $\overline{g^{-1}(X)\!\setminus\! L}\cap L=\{x_0,x_0+{\bf i}\}$. Suppose that the two components of $g\left(\left\{x_0+\frac{1}{2}{\bf i}+t: |t|\le1\right\}\right)\setminus\{x_0\}$ are contained in two components $U\ne V$ of $\hat{\bbC}\setminus X$. Then $g^{-1}(X)\setminus L^o$ is a locally connected continuum.
\end{prop}
\begin{proof}
By Proposition \ref{A4}, we just show that $g^{-1}(X)\setminus L^o$ is connected. Suppose on the contrary that it has two components, say $A$ and $B$ with $x_0\in A$. Let us decompose $g^{-1}(X)$ into the following union
 $$g^{-1}(X)=\underbrace{A\cup\left\{x_0+t{\bf i};0\leq t\leq \frac{1}{2}\right\}}_{\displaystyle=:A'}\;\bigcup\;\underbrace{B\cup\left\{x_0+t{\bf i};\frac{1}{2}\leq t\leq 1\right\}}_{\displaystyle=:B'}.
 $$
Note that $A'\cap B'=\{x_0+{\bf i}/2\}$ and that $x_0+{\bf i}/2$ belongs to the interior of $g^{-1}(U\cup V)\cup L^o$. By a separation theorem given in \cite[p.34]{Why64}, there exists a simple closed curve $J$  such that
$$J\cap \left(g^{-1}(X)\setminus L^o\right)=\{x_0+{\bf i}/2\},\;A'\setminus\{x_0+{\bf i}/2\}\subset\textrm{Int}(J),\;B'\setminus\{x_0+{\bf i}/2\}\subset\textrm{Ext}(J).
$$
Moreover, $J\setminus\{x_0+{\bf i}/2\}$, entirely contained in the complement of $g^{-1}(X)$, intersects both $g^{-1}(U)$ and $g^{-1}(V)$. It follows that $g(J)\setminus\{x_0\}$ is an open arc contained in the complement of $X$, intersecting both $U$ and $V$. This is impossible.
\end{proof}

\bibliographystyle{plain}
\bibliography{biblio}

\end{document}